\documentclass[12pt]{article}
\usepackage{}
\usepackage{amsthm,amsmath}
\usepackage{amscd}
 \usepackage{amssymb}
 \usepackage{hyperref}
 \usepackage[all]{xypic}
 \usepackage{mathtools}
 \usepackage[utf8]{inputenc}
  \newtheorem{lemma}{Lemma}[section]
 \newtheorem{corollary}[lemma]{Corollary}
 
 \newtheorem{theorem}[lemma]{Theorem}
 \newtheorem{proposition}[lemma]{Proposition}
 
 \newtheorem{remark}[lemma]{Remark}

\usepackage[utf8]{inputenc} 

\def\includegraphics{}

\usepackage{amscd}
\usepackage{amsmath,amssymb,mathrsfs}
\usepackage{hyperref}
\usepackage[all]{xypic}
\usepackage{times}
\usepackage{inputenc}
\usepackage{titlesec}

\newcommand{\dd}{\mathbb{D}}   \newcommand{\ba}{\mathcal {B}^{\alpha}}   
\newcommand{\Cu}{\mathcal{C}_\mu}       \newcommand{\sn}{\sum_{n=0}^{\infty}} 
\newcommand{\ii}{\int_{0}^{1}}    \newcommand{\hd}{H(\mathbb{D})}

 \usepackage{amsthm}
\newcommand{\comment}[1]{}
\newenvironment{proof of Theorem 1.1}{{\noindent\it Proof of Theorem 1.1}\quad}{\hfill $\square$\par}
\newenvironment{proof of Theorem 1.2}{{\noindent\it Proof of Theorem 1.2}\quad}{\hfill $\square$\par}
\newenvironment{proof of Theorem 1.3}{{\noindent\it Proof of Theorem 1.3}\quad}{\hfill $\square$\par}
\newenvironment{proof of Theorem 1.4}{{\noindent\it Proof of Theorem 1.4}\quad}{\hfill $\square$\par}
\newenvironment{proof of Theorem 1.5}{{\noindent\it Proof of Theorem 1.5}\quad}{\hfill $\square$\par}
\newenvironment{proof of Theorem 1.6}{{\noindent\it Proof of Theorem 1.6}\quad}{\hfill $\square$\par}
 
\advance\headheight1.15pt

\begin{document}

\baselineskip=8pt
\title{\fontsize{15}{0}\selectfont The Ces\`{a}ro-like operator on some analytic function spaces}
\author{\fontsize{11}{0}\selectfont
 Pengcheng Tang$^*$  \\ \fontsize{10}{0}\it{ College of Mathematics and Statistics, Hunan Normal University, Changsha, Hunan 410006, China}\\ 
}
\date{}
\maketitle
\thispagestyle{empty}
\begin{center}
\textbf{\underline{ABSTRACT}}
\end{center}
Let $\mu$ be a  finite positive Borel measure on the interval $[0, 1)$ and  $f(z)=\sum_{n=0}^{\infty}a_{n}z^{n} \in H(\mathbb{D})$. The Ces\`aro-like operator is defined by
$$
\Cu (f)(z)=\sum^\infty_{n=0}\left(\mu_n\sum^n_{k=0}a_k\right)z^n, \ z\in \dd,
$$
where, for $n\geq 0$, $\mu_n$ denotes the $n$-th  moment of the measure  $\mu$, that is,
$\mu_n=\int_{[0, 1)} t^{n}d\mu(t)$. Let $X$ and $Y$ be subspaces of $\hd$, the  purpose of this paper is to study the action of $\Cu$ on distinct pairs $(X, Y)$. The spaces considered in this paper are Hardy space $H^{p}(0<p\leq\infty)$, Morrey space $L^{2,\lambda}(0<\lambda\leq1)$, mean Lipschitz space, Bloch type space, etc.
\begin{flushleft}
\end{flushleft}
\begin{flushleft}
\end{flushleft}
\let\thefootnote\relax\footnote{$^*$Corresponding Author}
\let\thefootnote\relax\footnote{ Pengcheng Tang: www.tang-tpc.com@foxmail.com}

\vspace{1cm}
%

\section{Introduction} \label{Sec:Intro}

\ \ \ \ Let $\mathbb{D}=\{z\in \mathbb{C}:\vert z\vert <1\}$ denote the open unit disk of the complex plane $\mathbb{C}$ and $H(\mathbb{D})$ denote the space of all analytic functions in $\mathbb{D}$ and  $dA(z) =\frac{1}{\pi}dxdy$ the normalized area
Lebesgue measure.

 For $0<\alpha<\infty$, the Bloch-type space, denoted by $\ba$, is defined as
 $$\ba=\{f\in \hd:||f||_{\ba}=|f(0)|+\sup_{z\in\dd}(1-|z|^{2})^{\alpha}|f'(z)|<\infty\}.$$
If $\alpha=1$, then $\ba$ is just the classic Bloch space $\mathcal {B}$.

Let  $0<p\leq\infty$, the classical Hardy space $H^p$  consists of  those functions  $f\in H(\dd)$ for which
$$
||f||_{p}=\sup_{0\leq r<1} M_p(r, f)<\infty,
$$
where
$$
M_p(r, f)= \left(\frac{1}{2\pi}\int_0^{2\pi}|f(re^{i\theta})|^p d\theta \right)^{1/p}, \ 0<p<\infty,
$$
$$M_{\infty}(r, f)=\sup_{|z|=r}|f(z)|.$$

Let $I\subset \partial \mathbb{D}$ be an arc, and  $\vert I\vert $  denote the length of $I$. The Carleson square $S(I)$ is defined as
$$S(I)=\{re^{i\vartheta}:e^{i\vartheta}\in I,\ 1-\frac{\vert I\vert }{2\pi}\leq r<1\}.$$

Let $\mu$ be a positive Borel measure on $\mathbb{D}$. For $0\leq \beta<\infty$ and $0<t<\infty$, we
say that $\mu$ is a $\beta$-logarithmic $t$-Carleson measure (resp.a vanishing $\beta$-logarthmic $t$-Carleson measure) if
$$
\sup_{\vert I\vert \subset \partial\mathbb{D}}\frac{\mu(S(I))(\log\frac{2\pi}{\vert I\vert })^{\beta}}{\vert I\vert ^{t}}<\infty,\ \ \mbox{resp.}\ \ \lim_{\vert I\vert \rightarrow0}\frac{\mu(S(I))(\log\frac{2\pi}{\vert I\vert })^{\beta}}{\vert I\vert ^{t}}=0.
$$
See \cite{Zhao} for more about  logarithmic type Carleson measure.

A positive Borel measure $\mu$ on $[0,1)$ can be seen as a Borel measure on $\mathbb{D}$ by identifying it with the measure $\overline{\mu}$ defined by
$$
 \overline{\mu}(E)=\mu(E\cap [0,1)), \ \ \mbox{for any Borel subset }\ E \ \ \mbox{of}\  \ \mathbb{D}.
$$

In this way, a positive Borel measure $\mu$ on $[0,1)$ is a $\beta$-logarithmic $t$-Carleson measure if and only if there exists a constant $M>0$ such that
$$
\log^{\beta}\frac{e}{1-t}\mu([s,1))\leq M(1-s)^{t},\ \ 0\leq s<1.
$$

 Let $0 < \lambda\leq1$, the Morrey space $L^{2,\lambda}(\dd)$ is the set of all $f\in H^{2}$ such that
$$\sup_{I\subset \partial \dd}\left(\frac{1}{|I|^{\lambda}}\int_{I}|f(e^{i\theta})-f_{I}|^{2}d\theta\right)^{\frac{1}{2}}<\infty.$$
The space is $L^{2,\lambda}(\dd)$ a Banach space under the norm
$$||f||_{L^{2,\lambda}}=|f(0)|+ \sup_{I\subset \partial \dd}\left(\frac{1}{|I|^{\lambda}}\int_{I}|f(e^{i\theta})-f_{I}|^{2}d\theta\right)^{\frac{1}{2}}$$
It is well known that $L^{2,1}=BMOA$. The Morrey spaces increase when the parameter $\lambda$
decreases, so we have the following relation
$$BMOA\subseteq L^{2,\lambda_{2}}\subseteq L^{2,\lambda_{1}}\subseteq H^{2}, \ \ 0\leq \lambda_{1} \leq \lambda_{2}\leq 1. $$
For $0<\lambda\leq 1$ and  any function $f\in L^{2,\lambda}$, it has the following equivalent norm
$$||f||_{L^{2,\lambda}}\asymp |f(0)|+\sup_{w\in \dd}\left((1-|w|^{2})^{1-\lambda}\int_{\dd}|f'(z)|^{2}(1-|\sigma_{w}(z)|^{2})dA(z)\right)^{\frac{1}{2}},$$
where $\sigma_{w}$ stands for the M\"{o}bious transformation $\sigma_{w}(z)=\frac{w-z}{1-z\overline{w}}$.  See \cite{Liu1} for this characterization.

It is well known that functions  $f\in BMOA$ have logarithmic growth,
$$|f(z)|\leq C \log\frac{2}{1-|z|}.$$
This does not remain true for $f\in L^{2,\lambda}$ when $0<\lambda<1$.
Indeed, it follows Lemma 2 of \cite{Liu} that
$$H^{\frac{2}{1-\lambda}}\subseteq L^{2,\lambda} \subseteq H^{2},\ \ 0<\lambda<1.$$
It is known that for $0<\lambda<1$, $f\in L^{2,\lambda}$ satisfies
$$|f(z)|\lesssim\frac{||f||_{L^{2,\lambda}}}{(1-|z|)^{\frac{1-\lambda}{2}}},\ \ z\in\dd\eqno{(1.1)}$$
By (1.1) we have that  $L^{2,\lambda}\subseteq {B}^{\frac{3-\lambda}{2}}$ for all $0<\lambda\leq 1$. When $\lambda=1$,  it is obvious that the inclusion is strictly. For $0<\lambda<1$, the function $h(z)=\sum_{k=0}^{\infty}z^{2^{k}}\in \mathcal {B}\subsetneq \mathcal {B}^{\frac{3-\lambda}{2}}$ shows that the inclusion is also strictly. Since $h$  has a radial limit almost nowhere and hence $h\notin H^{p}$ for any $0<p<\infty$, this implies that $h\notin L^{2,\lambda}$.
The reader is referred to \cite{Wulan,Liu2,Wuz,Wuz1} for more about Morrey space.


Let  $1\leq p<\infty$ and $0<\alpha\leq 1$, the mean Lipschitz space $\Lambda^p_\alpha$ consists of those functions $f\in H(\dd)$ having a non-tangential limit  almost everywhere such that $\omega_p(t, f)=O(t^\alpha)$ as $t\to 0$. Here $\omega_p(\cdot, f)$ is the integral modulus of continuity of order $p$ of the function $f(e^{i\theta})$. It is  known (see \cite{du}) that $\Lambda^p_\alpha$ is a subset of $H^p$ and
$$\Lambda^p_\alpha=\left(f\in H(\dd):M_p(r, f')=O\left(\frac{1}{(1-r)^{1-\alpha}}\right), \ \ \mbox{as}\ r\rightarrow 1\right).$$
The space $\Lambda^p_\alpha$ is a Banach space with the norm $||\cdot||_{\Lambda^p_\alpha}$ given by
$$
\|f\|_{\Lambda^p_\alpha}=|f(0)|+\sup_{0\leq r<1}(1-r)^{1-\alpha}M_p(r, f').
$$
In \cite{bo}, Shapiro and Sledd proved that
$$\Lambda^{p}_{\frac{1}{p}}\subseteq BMOA. \ \ 1<p<\infty.$$
When $p=1$, the space $\Lambda^1_1$ is equivalent to the space
$$\left\{f\in H(\dd): \sup_{0\leq r<1}(1-r)M_{1}(f'',r)<\infty\right\}.\eqno{(1.2)}$$
See \cite{bla,ar} for more about  Lipschitz space and related analytic function spaces.

 For $f(z)=\sum_{n=0}^\infty \hat{f}(n)z^n\in \hd$, the Ces\`{a}ro operator $\mathcal {C}$ is defined  by
 $$
\mathcal {C}(f)(z)=\sum_{n=0}^\infty\left(\frac{1}{n+1}\sum_{k=0}^n \widehat{f}(k)\right)z^n=\int_{0}^{1}\frac{f(tz)}{1-tz}dt, \  z\in\dd.
$$
The Ces\`{a}ro operator   $\mathcal {C}$ is  bounded on $H^{p}$ for $0<p<\infty$. The case of $1<p<\infty$
follows from a result of Hardy on Fourier series \cite{ha} together with the Riesz
transform. Siskakis \cite{sis1} give an alternative proof of this result and to extend it to
$p = 1$ by using semigroups of  composition operators.  A direct proof of the boundedness on $H^1$ was given by Siskakis in \cite{sis2}. Miao
\cite{miao} proved the case $0 <p < 1$. Stempak \cite{ste} gave a proof valid for $ 0 < p \leq 2$.
 Andersen \cite{ad} and Nowak \cite{no} provided another proof valid for all $0<p<\infty$. In the case $p=\infty$, since $\mathcal {C}(1)(z)=\log\frac{1}{1-z}\notin H^{\infty}$, so that $\mathcal {C}(H^{\infty})\nsubseteq H^{\infty}$. Danikas  and  Siskakis \cite{sis4} proved  that $\mathcal {C}(H^{\infty})\nsubseteq BMOA$ and $\mathcal {C}(BMOA)\nsubseteq BMOA$. Ces\`{a}ro operator   $\mathcal {C}$ act on  weighted Bergman spaces, Dirichlet space  and general mixed normed spaces $H(p,q,\varphi)$ the reader is referred to \cite{sis2,v,ale,gla,shi}.

Recently, Galanopoulos, Girela and Merch\'an \cite{gla1} introduced a Ces\`aro-like operator $\Cu$ on $\hd$, which is a natural generalization of the classical Ces\`{a}ro operator $\mathcal {C}$.
They consider the following generalization: For a positive Borel measure
$\mu$ on the interval $[0, 1)$ they define the operator
 $$
\Cu (f)(z)=\sum^\infty_{n=0}\left(\mu_n\sum^n_{k=0}\widehat{f}(k)\right)z^n=\int_{0}^{1}\frac{f(tz)}{(1-tz)}d\mu(t), \  z\in\dd. \eqno{(1.3)}
$$
where $\mu_{n}$ stands for the  moment of order $n$ of $\mu$, that is,
  $\mu_{n}=\ii t^{n}d\mu(t)$. They studied the  operators $\Cu$  acting on distinct spaces of analytic functions(e.g. Hardy space, Bergman space, Bloch space, etc.).

 The  Ces\`aro-like operator $\Cu$ defined above has  attracted the interest  of many mathematicians.  Jin and Tang \cite{jin} studied  the boundedness(compactness) of  $\Cu$  from one Dirichlet-type space  into another one.
Bao, Sun and Wulan \cite{bao} studied the range of $\Cu$ acting on $H^{\infty}$. They proved that  $\Cu(H^{\infty})\subset \cap_{p>1}\Lambda^{p}_{\frac{1}{p}}$ if and only if   $\mu$ is a $1$-Carleson measure. This gives an answer to the question which was left open in \cite{gla1}.   In fact, they worked on a  more general version.  Just recently, Blasco \cite{blas}  used a different method to also get the same result.
Based on the previous results, it is natural to discuss
the range of $H^{\infty}$ under the action of $\Cu$ when $\mu$ is an $\alpha$-Carleson
measure with $0<\alpha < 1$. Furthermore,  what is the condition  for the measure  $\mu$ such that $\Cu(H^{p})\subset \cap_{q>1}\Lambda^{q}_{\frac{1}{q}}$?
We  shall  prove the following  general version of the results, which  give the answers to these questions. As consequences of our study, we may reproduce many of the known conclusions as well as obtain some new results.
 \begin{theorem}
Suppose  $0<p\leq \infty$, $0<\lambda \leq1$ and  $\mu$ is a  finite positive Borel measure on the interval $[0,1)$. Let $ X $   be  subspace of  $H(\dd)$ with $L^{2,\lambda}\subseteq X\subseteq\mathcal {B}^{\frac{3-\lambda}{2}}$ . Then   $\Cu(H^{p})\subseteq X$ if and only if $\mu$ is a $\frac{1+\lambda}{2}+\frac{1}{p}$-Carleson measure.
 \end{theorem}
\begin{theorem}
Suppose   $0<p<\infty$, $1<q<\infty$  and  $\mu$ is a  finite positive Borel measure on the interval $[0,1)$. Let $ X$ and $Y$ be subspaces
of  $H(\dd)$ such that $H^{p}\subseteq X\subseteq \mathcal {B}^{1+\frac{1}{p}}$ and  $\Lambda^{q}_{\frac{1}{q}}\subseteq Y\subseteq \mathcal {B}$. Then the following statements hold.
\\(1)\ The operator  $\Cu$ is bounded  from  $X$ into  $Y$ if and only if   $\mu$ is a  $1+\frac{1}{p}$-Carleson measure.
\\(2)\ If $\mu$ is a $1$-logarithmic $1+\frac{1}{p}$-Carleson measure, then $\Cu:X\rightarrow \Lambda^{1}_{1}$ is bounded.
 \end{theorem}
 If $1\leq p<\infty$, we   know that  $\Cu:BMOA\rightarrow \Lambda^{p}_{\frac{1}{p}}$ if and only if   $\mu$ is a $1$-logarithmic $1$-Carleson measure. Theorem 1.2 includes a characterization of those $\mu$ so that $\Cu$ maps  $L^{2,\lambda}$ into  $\Lambda^{p}_{\frac{1}{p}}$.

In \cite{gla1}, the authors proved that if  $X$ and $Y $ are spaces of holomorphic functions in the unit disc $\dd$, such that
$\Lambda^{2}_{\frac{1}{2}}\subseteq X,Y\subseteq \mathcal {B}$,
then $\Cu$ is a bounded operator from the space $X$ into the space $Y$ if and only if $\mu$ is a
$1$-logarithmic $1$-Carleson measure. Since $\Lambda^{2}_{\frac{1}{2}}\subseteq BMOA= L^{2,1}\subseteq \mathcal {B}$, so that $\Cu$ is a bounded operator from $X$ into the space $L^{2,1}$ if and only if $\mu$ is a $1$-logarithmic $1$-Carleson measure. 
  It is natural to  ask what's the condition for $\mu$ such that $\Cu$  is bounded from $X$ into $L^{2,\lambda}$? On the other hand, whether the  space $\Lambda^{2}_{\frac{1}{2}}$ can be extended to the space $\Lambda^{1}_{1}$?
We are now ready to state our next results, which generalized the previous mentioned results. Our results  also gives the range of $X$ under the action of $\Cu$ when $\mu$ is a $1$-logarithmic $s$-Carleson
measure with  $\frac{1}{2}< s<1$.
\begin{theorem}
Suppose   $0<\lambda\leq 1$ and  $\mu$ is a  finite positive Borel measure on the interval $[0,1)$. Let $ X$ and $Y$ be subspaces
of  $H(\dd)$ such that $\Lambda^{1}_{1}\subseteq X\subseteq \mathcal {B}$ and $L^{2,\lambda}\subset Y\subset \mathcal {B}^{\frac{3-\lambda}{2}}$. Then the following conditions are equivalent.
\\(1)\ The operator $\Cu$ is  bounded from $X$ into $Y$.
\\ (2)\ The measure $\mu$ is a $1$-logarithmic $\frac{1+\lambda}{2}$-Carleson measure.
 \end{theorem}

 \begin{theorem}
Suppose   $\mu$ is a  finite positive Borel measure on the interval $[0,1)$. Let $ X$ and $Y$ be subspaces
of  $H(\dd)$ such that $\Lambda^{1}_{1}\subseteq X,Y\subseteq \mathcal {B}$. Then the following conditions are equivalent.
\\(1)\ The operator $\Cu$ is  bounded from $X$ into  $Y$.
\\(2)\ The measure   $\mu$ is a $1$-logarithmic $1$-Carleson measure.
 \end{theorem}



 The boundedness  of the operator $\Cu$ acting on $BMOA$ has been studied in \cite{gla1,bao,blas}. The space of $BMOA$ is close related to the Morrey space $L^{2,\lambda}$.
Since the  Moreey space   $L^{2,\lambda}$ has showed up in a natural way in our work, it seems natural to study the action of the
operators $\Cu$ on the Moreey space   $L^{2,\lambda}$  for
general values of the parameters $\lambda$.
The following result gives a complete characterization of the boundedness  of $\Cu$ act between different Morrey spaces.
  Note that  the case  of  $\lambda_{1}=1$  is contained in Theorem 1.3.
\begin{theorem}
Suppose   $0<\lambda_{1}< 1$, $0<\lambda_{2}\leq 1$,  $\mu$ is a  finite positive Borel measure on the interval $[0,1)$.
Let $ X$ and $Y$ be subspaces of  $H(\dd)$ such that $L^{2,\lambda_{1}}\subseteq X\subseteq\mathcal {B}^{\frac{3-\lambda_{1}}{2}}$ and  $L^{2,\lambda_{2}}\subseteq Y\subseteq\mathcal {B}^{\frac{3-\lambda_{2}}{2}}$. Then the following statements are equivalent.
\\(1)\  The operator  $\Cu$ is bounded from  $X$ into  $Y$.
\\(2)\ The measure  $\mu$ is a  $1+\frac{\lambda_{2}-\lambda_{1}}{2}$-Carleson measure.
 \end{theorem}
 In section 2,  we shall give some  basic results that will be used in the proof.
 Section 3 will be devoted to present the proofs of Theorem 1.1--Theorem 1.5 and gives some relevant corollaries.
  It is necessary  to  clarify that the subspaces $X$ and $Y$ of $H(\dd)$ we shall
be dealing with are Banach spaces continuously embedded in $H(\dd)$, to prove that
the operator $\Cu$ is bounded from $X $ into $Y$ it suffices to show that
it maps $X$ into $Y$ by using the closed graph theorem.

Throughout the paper, the letter  $C$ will denote a positive constant which depends only upon the
displayed parameters (which sometimes will be omitted) but not necessarily the same at different occurrences. Furthermore, we  will use
the notation $Q_1\lesssim Q_{2}$  if there exists a constant $C$ such that $Q_1\leq C Q_{2}$, and $Q_1\gtrsim Q_{2}$ is
understood in an analogous manner. In particular, if  $Q_1\lesssim Q_{2}$  and $Q_1\gtrsim Q_{2}$, then we write $Q_1\asymp Q_{2}$ and say that $Q_{1}$ and $Q_{2}$ are equivalent. This notation has already been used above in the introduction.

 \section{Preliminary Results} \label{prelim}
\begin{lemma}\label{lm2.1}
Let  $0<\alpha<\infty$  and $f\in \ba$. Then for each $z\in \dd$, we have the following
inequalities:
\begin{equation*}
|f(z)|\lesssim
\begin{cases}
||f||_{\ba},  \text{ if  $0<\alpha<1;$}\\
||f||_{\ba}\log\frac{2}{1-|z|},  \text{ if  $\alpha=1;$ }\\
\frac{||f||_{\ba}}{(1-|z|)^{\alpha-1}}, \text{ if $\alpha>1.$}
\end{cases}
\end{equation*}
\end{lemma}
This well known Lemma can be found in \cite{zhu}.
\begin{lemma}\label{lm2.2}
Let  $\alpha>0$ and $f\in H(\mathbb{D})$, $f(z)=\sn \widehat{f}(n)z^{n}$, $\widehat{f}(n)\geq 0 $ for all $n\geq 0$. Then $f\in \mathcal {B}^{\alpha}$ if and only if
$$\sup_{n\geq 1}n^{-\alpha}\sum_{k=1}^{n}k\widehat{f}(k)<\infty.$$
\end{lemma}
This result follows from Corollary 3.2 in \cite{t} or Theorem 2.6 in \cite{Wulan1}.
\comment{
The following characterization of functions with nonnegative Taylor coefficients in
$L^{2,\lambda}$ is Theorem 26 in [].
\begin{lemma}Let $0<\lambda \leq 1$ and let $f(z)=\sum_{n=0}^\infty \hat{f}(n)z^n$ be an analytic function in $\dd$ with $ \hat{f}(n)\geq 0$ for all $n\geq 0$. Then
$f\in L^{2,\lambda}$ if and only if
$$
\sup_{0\leq r<1} \sum_{n=0}^\infty \frac{(1-r)^{2-\lambda}}{(n+1)^{2}}\left(\sum_{k=0}^n(k+1)\hat{f}(k+1)r^{n-k}\right)^2<\infty.
$$
In particular, if  $ \hat{f}(n)\geq 0$ for each $n\geq 0$ and \{$\hat{f}(n)$\}  non-increasing, then
$$f\in L^{2,\lambda}\ \mbox{if and only if}\ \hat{f}(k)\lesssim n^{-\frac{1+\lambda}{2}}. $$
\end{lemma}
For $0<\lambda<1$, the above Lemma and () show that the function $f(z)=(1-z)^{-1-\lambda}$ belongs to the Morrey space $L^{2,\lambda}$ and attained maximum growth.
}
\begin{lemma}\label{lm2.3}
Let $0<s<\infty$ and $\mu$ be a finite positive Borel measure on the interval $[0,1)$. Then the following statements hold:
\\ (1)\ $\mu$ is an $s$-Carleson measure if and only if $\mu_{n}= O(\frac{1}{n^{s}})$.
\\ (2)\  $\mu$ is a vanishing $s$-Carleson measure if and only if $\mu_{n}= o(\frac{1}{n^{s}})$.
\end{lemma}
 This Lemma follows from Theorem 2.1  and Theorem 2.4 in \cite{bao1}.

 The following integral estimates are useful.  We only list the required ones. See \cite{z} for the detailed proofs and other cases.
\begin{lemma}\label{lm2.4}
Suppose that $r\geq 0,t\geq 0,\delta>-1,k \geq 0.$ Let
$$J_{w,a}=\int_{\dd}\frac{(1-|z|^{2})^{\delta}}{|1-z\overline{w}|^{t}|1-z\overline{a}|^{r}}\log^{k}\frac{e}{1-|z|^{2}}dA(z),\ \ w,a \in \dd.$$
(1)\ If  $t+r-\delta>2$, $t-\delta<2$ and $r-\delta<2$, then
$$J_{w,a} \asymp \frac{1}{|1-\langle w,a\rangle|^{t+r-\delta-2}}\log^{k}\frac{e}{|1-\langle w,a\rangle|}.$$
(2)\ If $t-\delta>2>r-\delta$, then
$$J_{w,a} \asymp \frac{1}{(1-|w|^{2})^{t-\delta-n-1}|1-\langle w,a\rangle|^{r}}\log^{k}\frac{e}{1-|w|^{2}}.$$
\end{lemma}
We  also need the following estimates. (See e.g. Theorem 1.12 in \cite{zhu})
\begin{lemma}\label{lm2.5}
 Let $\alpha$ be any real number and $z\in \dd$.  Then
$$
\int^{2\pi}_0\frac{d\theta}{|1-ze^{-i\theta}|^{\alpha}}\asymp
\begin{cases}1 & \enspace \text{if} \ \ \alpha<1,\\
                     \log\frac{2}{1-|z|^2} & \enspace  \text{if} \ \  \alpha=1,\\
                     \frac{1}{(1-|z|^2)^{\alpha-1}} & \enspace \text{if}\ \  \alpha>1,
                   \end{cases}
$$
\end{lemma}
The following result is known to experts. We give a  detailed proof by using  the  integral estimates with  double variable points. These integral estimates are practical and have its own interests. The reader is referred to  \cite{z,lzx,xt}  for various integral estimates.
\begin{lemma}\label{lm2.6}
Let $0<\lambda<1$, then for any  $c\leq \frac{1-\lambda}{2}$,  we have
$$f(z)=\frac{1}{(1-z)^{c}}\in L^{2,\lambda}.$$
\end{lemma}
\begin{proof}
It is suffices to prove the case of  $c=\frac{1-\lambda}{2}$. For $0<r<1$ and $w\in \dd$, by Proposition 3.1-(7) in \cite{xt}  we have
$$\int_{0}^{2\pi}\frac{d\theta}{|1-re^{i\theta}|^{3-\lambda}|1-r\overline{w}e^{i\theta}|^{2}}\asymp \frac{1}{(1-r)^{2-\lambda}|1-r^{2}\overline{w}|^{2}}+\frac{1}{(1-r^{2}|w|)|1-r^{2}\overline{w}|^{3-\lambda}}.$$
It is easy to check that
$$\frac{1}{|1-r^{2}\overline{w}|^{2}}\lesssim\frac{1}{(1-r|w|)^{2}}\ \ \mbox{and}\ \frac{1-r}{(1-r^{2}|w|)|1-r^{2}\overline{w}|^{3-\lambda}}\lesssim \frac{(1-r)^{\lambda-1}}{(1-r|w|)^{2}}.$$
Using  the  polar coordinate formula  and above  inequalities we get
 \[ \begin{split}
  ||f||_{L^{2.\lambda}}&\asymp \sup_{w\in \dd}\left((1-|w|^{2})^{1-\lambda}\int_{\dd}|f'(z)|^{2}(1-|\sigma_{w}(z)|^{2})dA(z)\right)^{\frac{1}{2}}\\
& \lesssim  \sup_{w\in \dd}(1-|w|^{2})^{\frac{2-\lambda}{2}}\left(
\int_{0}^{1}(1-r)\int_{0}^{2\pi}\frac{d\theta}{|1-re^{i\theta}|^{3-\lambda}|1-r\overline{w}e^{i\theta}|^{2}}dr\right)^{\frac{1}{2}}\\
& \asymp\sup_{w\in \dd}(1-|w|^{2})^{\frac{2-\lambda}{2}}\left(
\int_{0}^{1}\frac{(1-r)^{\lambda-1}}{|1-r^{2}\overline{w}|^{2}}dr+\int_{0}^{1}\frac{1-r}{(1-r^{2}|w|)|1-r^{2}\overline{w}|^{3-\lambda}}dr\right)^{\frac{1}{2}}\\
& \lesssim \sup_{w\in \dd}(1-|w|^{2})^{\frac{2-\lambda}{2}}\left(
\int_{0}^{1}\frac{(1-r)^{\lambda-1}}{(1-r|w|)^{2}}dr\right)^{\frac{1}{2}}\\
&\lesssim 1.
  \end{split} \]
 The last step above we have used the integral estimate
 $$\int_{0}^{1}\frac{(1-r)^{\lambda-1}}{(1-r|w|)^{2}}dr \asymp \frac{1}{(1-|w|)^{2-\lambda}},$$
which can be found in the literature \cite{z1}.
\end{proof}

\section{Proofs of the main results}
First, we give some characterizations of positive Borel measures $\mu$
on $[0, 1)$ as logarithmic type Carleson measures, this will be used in our proofs.
\begin{proposition}\label{pro3.1}
Suppose  $\beta>0$, $\gamma\geq 0$, $0\leq  q<s<\infty$ and  $\mu$ is a finite positive  Borel measure on $[0,1)$. Then the following conditions are equivalent:
\begin{enumerate}
  \item [(1)] $\mu$ is a $\gamma$-logarithmic $s$-Carleson measure;
   \item [(2)]  $$
S_{1}:=\sup_{w\in\dd}\int_{0}^{1}\frac{(1-|w|)^{\beta}\log^{\gamma}\frac{e}{1-|w|}}{(1-t)^{q}(1-|w|t)^{s+\beta-q}}d\mu(t)<\infty;
$$
  \item [(3)]$$
S_{2}:=\sup_{w\in\dd}\int_{0}^{1}\frac{(1-|w|)^{\beta}\log^{\gamma}\frac{e}{1-|w|}}{(1-t)^{q}|1-wt|^{s+\beta-q}}d\mu(t)<\infty.
$$
\item [(4)]$$
S_{3}:=\sup_{w\in\dd}\int_{0}^{1}\frac{(1-|w|)^{\beta}\log^{\gamma}\frac{e}{1-t}}{(1-t)^{q}(1-|w|t)^{s+\beta-q}}d\mu(t)<\infty.
$$
\end{enumerate}
\end{proposition}
\begin{proof}
The proof of $(2)\Rightarrow(1)$ is straightforward. In fact,
 \[ \begin{split}
 S_{1}&:=\sup_{w\in\dd}\int_{0}^{1}\frac{(1-|w|)^{\beta}\log^{\gamma}\frac{e}{1-|w|}}{(1-t)^{q}(1-|w|t)^{s+\beta-q}}d\mu(t)\\
 & \geq \int_{|w|}^{1}\frac{(1-|w|)^{\beta}\log^{\gamma}\frac{e}{1-|w|}}{(1-t)^{q}(1-|w|t)^{s+\beta-q}}d\mu(t)\\
 & \gtrsim \frac{\mu(|w|,1)\log^{\gamma}\frac{e}{1-|w|}}{(1-|w|)^{s}}.
  \end{split} \]
  This  finish the proof of $(2)\Rightarrow(1)$. Similarly, we may obtain $(3)\Rightarrow(1)$ and  $(4)\Rightarrow(1)$. Since $(2)\Rightarrow(3)$ is obvious, to complete the proof we have to prove that $(1)\Rightarrow(2)$ and $(1)\Rightarrow(4)$.

$(1)\Rightarrow(2)$. The proof of this implication  follows closely the arguments of the proof of Proposition 2.1 in  \cite{bao}. We include a detailed proof for completeness.

It is suffices  to consider  the case $w\in \dd$ with $\frac{1}{2}\leq |w|<1$ and $q>0$.
For every positive integer $n\geq1$, let
$$
	Q_{0}(w)=\varnothing , Q_n(w)=\{t\in[0,1): 1-2^n(1-|w|)\leq t<1\} .
$$
Let $n_w$ be the minimal integer such that $1-2^{n_w}(1-|w|)\leq 0$. Then $Q_n(w)=[0, 1)$ when $n\geq n_w$.
If   $ t\in Q_1(w)$, then
$$1-|w| \leq  1-|w|t.$$
Also, for $2\leq n\leq  n_w$ and $t\in Q_n(w)\backslash Q_{n-1}(w)$, we have
$$
(2^{n-1}-1)(1-|w|)=|w|-(1-2^{n-1}(1-|w|))\leq |w|-t\leq 1-|w|t.
$$
Notice that   $\beta>0$, $\gamma\geq 0$,  $0<q<s<\infty$ and  $\mu$ is a $\gamma$-logarithmic $s$-Carleson measure, these together with above inequalities we have
{\small {\small
 \[ \begin{split}
& \ \ \ \ \int_{0}^{1}\frac{\log^{\gamma}\frac{e}{1-|w|}(1-|w|)^{\beta}}{(1-t)^q(1-|w|t)^{s+\beta-q}}d\mu(t)\\
& =\sum^{n_w}_{n=1}\int_{Q_n(w)\backslash Q_{n-1}(w)}\frac{\log^{\gamma}\frac{e}{1-|w|}(1-|w|)^\beta}{(1-t)^q(1-|w|t)^{s+\beta-q}}d\mu(t)\\
&\lesssim \sum^{n_w}_{n=1}\frac{\log^{\gamma}\frac{e}{1-|w|}(1-|w|)^{q-s}}{2^{n(s+\beta-q)}}\int_{Q_n(w)\backslash Q_{n-1}(w)} \frac{1}{(1-t)^{q}}d\mu(t)\\
&\lesssim   \sum^{n_w}_{n=1}\frac{\log^{\gamma}\frac{e}{1-|w|}(1-|w|)^{q-s}}{2^{n(s+\beta-q)}}\int_0^\infty x^{q-1}\mu\big(\big\{t\in[1-2^n(1-|w|),1): 1-\frac{1}{x}<t\big\}\big)dx\\
&\asymp \sum^{n_w}_{n=1}\frac{\log^{\gamma}\frac{e}{1-|w|}(1-|w|)^{q-s}}{2^{n(s+\beta-q)}}\int_0^{\frac{1}{2^n(1-|w|)}}
x^{q-1}\mu\big([1-2^n(1-|w|),1))dx\\
&\ \ \ \ 	+  \sum^{n_w}_{n=1}\frac{\log^{\gamma}\frac{e}{1-|w|}(1-|w|)^{q-s}}{2^{n(s+\beta-q)}}
\int_{\frac{1}{2^n(1-|w|)}}^\infty x^{q-1}\mu\big(\big[1-\frac{1}{x},1\big)\big)dx\\
&\lesssim \sum^{n_w}_{n=1}\frac{\log^{\gamma}\frac{e}{1-|w|}(1-|w|)^{q-s}}{2^{n(s+\beta-q)}}\left(
\frac{2^{ns}(1-|w|)^{s}}{\log^{\gamma}\frac{e}{2^{n}(1-|w|)}}\int_0^{\frac{1}{2^n(1-|w|)}}x^{q-1}dx+\int_{\frac{1}{2^n(1-|w|)}}^\infty \frac{\log^{-\gamma}ex}{x^{s+1-q}}dx\right)\\
&\lesssim  \sum^{n_w}_{n=1} \frac{1}{2^{\beta n}}\frac{\log^{\gamma}\frac{e}{1-|w|}}{\log^{\gamma}\frac{e}{2^{n}(1-|w|)}}\lesssim  \sum^{n_w}_{n=1} \frac{1}{2^{\beta n}}\left(1+\frac{n^{\gamma}\log 2}{\log^{\gamma}\frac{2}{2^{n}(1-|w|)}}\right)
\lesssim \sum^{n_w}_{n=1} \frac{n^{\gamma}}{2^{\beta n}}\lesssim 1.
	 \end{split} \]}}
This implies that
$$
S_{1}:=\sup_{w\in \dd}\int_{0}^{1}\frac{(1-|w|)^\beta}{(1-t)^q(1-|w|t)^{s+\beta-q}}d\mu(t)<\infty.
$$

$(1)\Rightarrow(4)$.   We only need consider the case  of  $\gamma>0$. For $0<\delta<s-q$,  let
$$f(t)=(1-t)^{\delta}\log^{\gamma}\frac{e}{1-t}, \ \  0\leq t<1.$$
It is known  that $f$ is a normal function on $[0,1)$. Furthermore, we may choosing  $b=\delta$ and $0<a=\varepsilon<\delta$ such that
$$\frac{f(t)}{(1-t)^{b}}\mbox{is increasing},\  \frac{f(t)}{(1-t)^{a}}\mbox{is decreasing}, \ \mbox{as}\ t\rightarrow 1^{-}. $$
Hence, it follows form  Lemma 2.2 in \cite{zl} that
$$\frac{f(t)}{f(r)}\lesssim  \left(\frac{1-t}{1-r}\right)^{\varepsilon}+\left(\frac{1-t}{1-r}\right)^{\delta}\eqno{(3.1)}$$
for all $0<t,r<1$. Bearing in mind that $(1)\Leftrightarrow(2)$ we have proved already.
 By (3.1) we have
 \[ \begin{split}
&\ \ \ \ \int_{0}^{1}\frac{(1-|w|)^{\beta}\log^{\gamma}\frac{e}{1-t}}{(1-t)^{q}(1-|w|t)^{s+\beta-q}}d\mu(t)\\
&= \int_{0}^{1}\frac{(1-|w|)^{\beta+\delta}\log^{\gamma}\frac{e}{1-|w|}}{(1-t)^{q+\delta}(1-|w|t)^{s+\beta-q}}\cdot \frac{f(t)}{f(|w|)}d\mu(t)\\
&\lesssim \int_{0}^{1}\frac{(1-|w|)^{\beta+\delta}\log^{\gamma}\frac{e}{1-|w|}}{(1-t)^{q+\delta}(1-|w|t)^{s+\beta-q}} \left\{\left(\frac{1-t}{1-|w|}\right)^{\varepsilon}+\left(\frac{1-t}{1-|w|}\right)^{\delta}\right\}d\mu(t)\\
&\lesssim \int_{0}^{1}\frac{(1-|w|)^{\beta}\log^{\gamma}\frac{e}{1-|w|}}{(1-t)^{q}(1-|w|t)^{s+\beta-q}} d\mu(t)+ \int_{0}^{1}\frac{(1-|w|)^{\beta+\delta-\varepsilon}\log^{\gamma}\frac{e}{1-|w|}}{(1-t)^{q+\delta-\varepsilon}(1-|w|t)^{s+\beta-q}} d\mu(t)\\
&\lesssim 1.
  \end{split} \]
  This gives $(4)$.
\end{proof}
\begin{remark}
For $\gamma\in \mathbb{R}$ and $0<s<\infty$, we may prove the following result in a same way.
$$\sup_{t\in [0,1)}\frac{\log^{\gamma}\frac{e}{1-t}\mu([t,1))}{(1-t)^{s}}<\infty\Leftrightarrow(3.1)\Leftrightarrow (3.2)\Leftrightarrow (3.3).$$
\end{remark}
We  now present the proofs of Theorems 1.1--Theorem 1.5.

\begin{proof of Theorem 1.1}
 (1).
 If $\Cu(H^{p})\subseteq X$,  take
 $$f_{a}(z)=\frac{(1-a)}{(1-az)^{1+\frac{1}{p}}},\ \ 0<a<1.$$
 Then $f_{a}\in H^{p}$ for all $0<p\leq \infty$ and $\sup_{0<a<1}||f_{a}||_{p}\lesssim 1$. This implies that
  $$\Cu(f_{a})\in X\subseteq \mathcal {B}^{\frac{3-\lambda}{2}}.$$
  It is easy to see that
  $$\Cu(f_{a})'(z)=\int_{0}^{1}\frac{tf'_{a}(tz)}{(1-tz)}d\mu(t)+\int_{0}^{1}\frac{tf_{a}(tz)}{(1-tz)^{2}}d\mu(t).$$
  Since $\Cu(f_{a})\in X \subseteq \mathcal {B}^{\frac{3-\lambda}{2}}$,  it follows from Lemma \ref{lm2.1}  that
 $$|\Cu(f_{a})'(a)|\lesssim \frac{1}{(1-a)^{\frac{3-\lambda}{2}}},\ \  a\in(0,1).$$
 Then it follows that, for $\frac{1}{2}<a<1$,
 \[ \begin{split}
 \frac{1}{(1-a)^{\frac{3-\lambda}{2}}}&\gtrsim \left|\int_{0}^{1}\frac{(1+\frac{1}{p})ta(1-a)}{(1-ta)(1-ta^{2})^{2+\frac{1}{p}}}d\mu(t)+\int_{0}^{1}\frac{t(1-a)}{(1-ta)^{2}(1-ta^{2})^{1+\frac{1}{p}}}d\mu(t)\right|\\
 &\gtrsim \int_{a}^{1}\frac{1}{(1-ta^{2})^{2+\frac{1}{p}}}d\mu(t)\\
& \gtrsim\frac{\mu([a,1))}{(1-a)^{2+\frac{1}{p}}}.
 \end{split} \]
 This gives that
 $$\mu([a,1))\lesssim (1-a)^{\frac{1+\lambda}{2}+\frac{1}{p}}\ \ \mbox{for all }\  \frac{1}{2}<a<1.$$
 This implies that $\mu$ is a  $\frac{1+\lambda}{2}+\frac{1}{p}$-Carleson measure.

 On the other hand,  suppose $\mu$ is a $\frac{1+\lambda}{2}+\frac{1}{p}$-Carleson measure. Let $L^{2,\lambda}\subseteq X\subseteq\mathcal {B}^{\frac{3-\lambda}{2}}$, to prove $\Cu(H^{p})\subseteq X$ it is sufficient to prove that $\Cu:H^{p}\rightarrow L^{2,\lambda}$ is bounded. Without loss of generality, we may assume $f\in H^{p}$ and $f(0)=0$. By (1.3), we know that
$$
	\Cu (f)'(z)=\int_{0}^{1}\frac{tf'(tz)}{(1-tz)}d\mu(t)+ \int_{0}^{1}\frac{tf(tz)}{(1-tz)^{2}}d\mu(t), \quad z\in\dd.
$$
Let
\[\delta_{p}=\left\{
  \begin{array}{cc}
 \displaystyle{\frac{1}{p}} & \ \ \ \ \ \ \ \ \displaystyle{0<p<\infty};\\
 \displaystyle{0}& \ \ \ \ \ \ \ \ \ \displaystyle{p=\infty.}
 \end{array}\right.\]
It is known that (see e.g.  page 36  in \cite{du})
$$|f(z)|\lesssim \frac{||f||_{p}}{(1-|z|)^{\delta_{p}}},$$
and hence  $$|f'(z)|\lesssim \frac{||f||_{p}}{(1-|z|)^{1+\delta_{p}}}.$$
It follows that
\[ \begin{split}
|\Cu (f)'(z)|&\leq \int_{0}^{1}\frac{|tf'(tz)|}{|1-tz|}d\mu(t)+ \int_{0}^{1}\frac{|tf(tz)|}{|1-tz|^{2}}d\mu(t)\\
&\leq ||f||_{p}\int_{0}^{1}\frac{d\mu(t)}{|1-tz|(1-t|z|)^{1+\delta_{p}}}+ ||f||_{p} \int_{0}^{1}\frac{d\mu(t)}{(1-t|z|)^{\delta_{p}}|1-tz|^{2}}
  \end{split} \]
$$\lesssim ||f||_{p} \int_{0}^{1}\frac{d\mu(t)}{|1-tz|(1-t|z|)^{1+\delta_{p}}}.\  \ \ \ \ \ \ \ \ \ \ \ \ \ \ \ \ \ \ \ \ \ \ \ \ \ \ \ \ \ \ \  \eqno{(3.2)}
$$
  Since $0<\lambda<1$, we can choose a positive real  number $1-\lambda<\sigma<1$ such that
  $$\frac{1}{(1-t|z|)^{2+2\delta_{p}}}\leq \frac{1}{(1-t)^{2+2\delta_{p}-\sigma}(1-|z|)^{\sigma}}.\eqno{(3.3)}$$
By (3.2)  and Minkowski's inequality, (3.3), Lemma \ref{lm2.4} and Proposition \ref{pro3.1}, we get
  \[ \begin{split}
 &||\Cu(f)||_{L^{2.\lambda}}\asymp \sup_{w\in \dd}\left((1-|w|^{2})^{1-\lambda}\int_{\dd}|\Cu(f)'(z)|^{2}(1-|\sigma_{w}(z)|^{2})dA(z)\right)^{\frac{1}{2}}\\
&  \lesssim  ||f||_{p}\sup_{w\in \dd}(1-|w|^{2})^{\frac{2-\lambda}{2}}\left(\int_{\dd}
  \left(\int_{0}^{1}\frac{d\mu(t)}{|1-tz|(1-t|z|)^{1+\delta_{p}}}\right)^{2}\frac{1-|z|^{2}}{|1-z\overline{w}|^{2}}dA(z)\right)^{\frac{1}{2}}\\
 & \leq  ||f||_{p}\sup_{w\in \dd}(1-|w|^{2})^{\frac{2-\lambda}{2}}\int_{0}^{1}\left(\int_{\dd}\frac{(1-|z|^{2})dA(z)}{(1-t|z|)^{2(1+\delta_{p})}|1-tz|^{2}|1-z\overline{w}|^{2}}\right)^{\frac{1}{2}}d\mu(t)\\
 & \lesssim  ||f||_{p}\sup_{w\in \dd}(1-|w|^{2})^{\frac{2-\lambda}{2}}\int_{0}^{1}\frac{1}{(1-t)^{1+\delta_{p}-\frac{\sigma}{2}}}
 \left(\int_{\dd}\frac{(1-|z|)^{1-\sigma}dA(z)}{|1-tz|^{2}|1-z\overline{w}|^{2}}\right)^{\frac{1}{2}}d\mu(t)\\
 & \asymp  ||f||_{p}\sup_{w\in \dd} \int_{0}^{1}\frac{(1-|w|^{2})^{\frac{2-\lambda}{2}}}{(1-t)^{1+\delta_{p}-\frac{\sigma}{2}}|1-tw|^{\frac{1+\sigma}{2}}}d\mu(t)\\
 & \lesssim ||f||_{p}.
    \end{split} \]
    Therefore, $\Cu:H^{p}\rightarrow L^{2,\lambda}$ is bounded.

  The proof is complete.
\end{proof of Theorem 1.1}

\begin{corollary}
Suppose $0<\lambda\leq 1$ and  $\mu$ is a  finite positive Borel measure on the interval $[0,1)$. Then  $\Cu: H^{\infty}\rightarrow L^{2,\lambda}$ is bounded if and only if $\mu$ is a $\frac{1+\lambda}{2}$-Carleson measure.
\end{corollary}
\begin{remark}
If $\frac{1}{2}<\alpha\leq1$, then Corollary 3.3  show that   $\mu$ is an $\alpha$-Carleson measure if and only if $\Cu(H^{\infty})\subseteq L^{2,2\alpha-1}$. When $0<\alpha\leq \frac{1}{2}$ and $\mu$ is an $\alpha$-Carleson measure, by  Proposition \ref{pro3.1} we have
  \[ \begin{split}
  \sup_{z\in\dd}(1-|z|^{2})^{2-\alpha}|\Cu(f)'(z)|&\lesssim ||f||_{H^\infty}\sup_{z\in\dd}\int_{0}^{1}\frac{(1-|z|^{2})^{2-\alpha}d\mu(t)}{(1-t|z|)|1-tz|}\\
  &\lesssim ||f||_{H^\infty}\sup_{z\in\dd}\int_{0}^{1}\frac{(1-|z|^{2})^{1-\alpha}d\mu(t)}{|1-tz|}\\
  &\lesssim ||f||_{H^\infty}.
     \end{split} \]
This yields  that $\Cu(H^{\infty})\subseteq \mathcal {B}^{2-\alpha}$.$\square$

\end{remark}
For $ 2<p\leq \infty$, it follows from Theorem 9 in \cite{Wuz1} that  the Cesr\`{a}o operator $\mathcal {C}$ is bounded from $H^{p}$ to $L^{2,1-\frac{2}{p}}$. As a consequence of Theorem 1.1, we have the following result.
\begin{corollary}
Suppose   $2<p\leq \infty$ and  $\mu$ is a  finite positive Borel measure on the interval $[0,1)$. Then $\Cu: H^{p}\rightarrow L^{2,1-\frac{2}{p}}$ is bounded if and only if $\mu$ is a  $1$-Carleosn measure.
\end{corollary}

\begin{proof of Theorem 1.2}
 (1). The proof of necessity is similar to that Theorem 1.1 and hence omitted. For  the sufficiency, it is suffices to show that $\Cu(\mathcal {B}^{1+\frac{1}{p}})\subseteq \Lambda^{q}_{\frac{1}{q}}$ when $\mu$ is an $1+\frac{1}{p}$-Carleson measure.

Notice that  (3.2) is  remain valid for all $f\in \mathcal {B}^{1+\frac{1}{p}}$. By (3.2) and  the Minkowski inequality,  Lemma \ref{lm2.5} and Proposition \ref{pro3.1}  we have
\begin{align*}
&\sup_{0<r<1}(1-r)^{1-\frac1q}\left(\frac{1}{2\pi}\int_0^{2\pi}|\Cu (f)'(re^{i\theta})|^q d\theta\right)^{\frac1q}\nonumber \\
\lesssim &\|f\|_{\mathcal {B}^{1+\frac{1}{p}}} \sup_{0<r<1}(1-r)^{1-\frac1q} \left(\frac{1}{2\pi}\int_0^{2\pi}\left(\int^{1}_{0} \frac{d\mu(t)}{|1-tre^{i\theta}|(1-tr)^{1+\delta_{p}}} \right)^q d\theta\right)^{\frac1q} \\
\lesssim&\|f\|_{\mathcal {B}^{1+\frac{1}{p}}}  \sup_{0<r<1}(1-r)^{1-\frac1q} \int^{1}_{0} \left(\frac{1}{2\pi}\int_0^{2\pi} \frac{1}{|1-tre^{i\theta}|^{q}(1-tr)^{q(1+\delta_{p})}}  d\theta\right)^{\frac1q}  d\mu(t)\\
\lesssim &\|f\|_{\mathcal {B}^{1+\frac{1}{p}}}  \sup_{0<r<1} \int^{1}_{0} \frac{(1-r)^{1-\frac1q}}{(1-tr)^{2+\delta_{p}-\frac{1}{q}}}d\mu(t)\\
\lesssim &\|f\|_{\mathcal {B}^{1+\frac{1}{p}}}.
\end{align*}
This gives $\Cu: \mathcal {B}^{1+\frac{1}{p}}\rightarrow \Lambda^{q}_{\frac{1}{q}}$ is bounded.

 (2). Suppose $\mu$ is a $1$-logarithmic $1+\frac{1}{p}$-Carleson measure. Let $H^{p} \subseteq X\subseteq\mathcal {B}^{1+\frac{1}{p}}$ and  $f\in X$,
then  $f\in X \subseteq\mathcal {B}^{1+\frac{1}{p}}$.
Using the integral representation of $\Cu$ we see that
$$\Cu(f)''(z)=\int_{0}^{1}\frac{t^{2}f''(tz)}{1-tz}d\mu(t)+2\int_{0}^{1}\frac{t^{2}f'(tz)}{(1-tz)^{2}}d\mu(t)+2\int_{0}^{1}\frac{t^{2}f(tz)}{(1-tz)^{3}}d\mu(t).\eqno{(3.4)}$$
 It follows from Lemma \ref{lm2.1} we have that
\[ \begin{split}
|\Cu(f)''(z)|&\lesssim ||f||_{\mathcal {B}^{1+\frac{1}{p}}}\int_{0}^{1}\left(\frac{(1-t|z|)^{-2-\frac{1}{p}}}{|1-tz|}+\frac{(1-t|z|)^{-1-\frac{1}{p}}}{|1-tz|^{2}}+\frac{(1-t|z|)^{-\frac{1}{p}}}{|1-tz|^{3}}\right)d\mu(t)\\
&\lesssim ||f||_{\mathcal {B}^{1+\frac{1}{p}}}\int_{0}^{1}\frac{1}{(1-t|z|)^{2+\frac{1}{p}}|1-tz|}d\mu(t).
  \end{split} \]
 By  Fubini's theorem, Lemma \ref{lm2.5} and Proposition \ref{pro3.1}, we have
    \[ \begin{split}
 & \ \ \ \ \sup_{0\leq r<1}(1-r)M_{1}(\Cu(f)'',r)\\
  & \lesssim  ||f||_{\mathcal {B}^{1+\frac{1}{p}}}\sup_{0\leq r<1}(1-r)\int_{0}^{2\pi}
  \int_{0}^{1}\frac{d\mu(t)}{(1-tr)^{2+\frac{1}{p}}|1-tre^{i\theta}|}d\theta\\
   &\lesssim ||f||_{\mathcal {B}^{1+\frac{1}{p}}} \sup_{0\leq r<1} \int_{0}^{1}\frac{1-r}{(1-tr)^{2+\frac{1}{p}}}\int_{0}^{2\pi}\frac{d\theta}{|1-tre^{i\theta}|}d\mu(t)\\
   &\lesssim ||f||_{\mathcal {B}^{1+\frac{1}{p}}} \sup_{0\leq r<1}\int_{0}^{1}\frac{(1-r)\log\frac{e}{1-tr}}{(1-tr)^{2+\frac{1}{p}}}d\mu(t)\\
&\lesssim ||f||_{\mathcal {B}^{1+\frac{1}{p}}}\lesssim ||f||_{X}.
    \end{split} \]
    This gives $\Cu:X\rightarrow \Lambda^{1}_{1}$ is bounded.
\end{proof of Theorem 1.2}
Theorem 1.1 and Theorem 1.2  lead to the following result.
\begin{corollary}
Suppose   $0<p\leq\infty$, $1<q<\infty$  and  $\mu$ is a  finite positive Borel measure on the interval $[0,1)$. Let $X$ be a subspace of $\hd$ with $\Lambda^{q}_{\frac{1}{q}}\subseteq X\subseteq \mathcal {B}$. Then $\Cu: H^{p}\rightarrow X$ is bounded if and only if $\mu$ is a $1+\frac{1}{p}$-Carleson measure.
\end{corollary}
\begin{remark}
In \cite{blas},  Blasco proved that $\mathcal {C}_{\eta}: H^{1}\rightarrow \Lambda^{2}_{\frac{1}{2}}$ is bounded if and only if
$$\sup_{n\geq 0}(n+1)^{3}\sum_{k=n}^{\infty}|\eta_{k}|^{2}<\infty,\eqno{(3.5)}$$
where $\eta$ is a complex Borel measure on $[0,1)$. See  Theorem 3.7 in \cite{blas} for the detailed.
If  $\mu$ is a positive Borel measure on $[0,1)$, then Corollary 3.6 shows that $\mathcal {C}_{\mu}: H^{1}\rightarrow \Lambda^{2}_{\frac{1}{2}}$ is bounded if and only if  $\mu$ is a $2$-Carleosn measure. The condition (3.5)  is  equivalent to   $\mu$ is a $2$-Carleosn measure when $\mu$ is a positive Borel measure on $[0,1)$. In fact,
$$\infty> \sup_{n\geq 0}(n+1)^{3}\sum_{k=n}^{\infty}|\mu_{k}|^{2}
 \gtrsim  \sup_{n\geq 0}(n+1)^{3}\sum_{k=n}^{2n}\mu_{k}^{2}
 \gtrsim  \sup_{n\geq 0}(n+1)^{4}\mu^{2}_{2n}.$$
On the other hand, if  $\mu$ is a $2$-Carleosn measure,  we have
\[ \begin{split}
 \sup_{n\geq 0}(n+1)^{3}\sum_{k=n}^{\infty}|\mu_{k}|^{2} &\lesssim  \sup_{n\geq 0}(n+1)^{3}\sum_{k=n}^{\infty}\frac{1}{(k+1)^{4}}\\
 &\lesssim  \sup_{n\geq 0}(n+1)^{3}\int_{n+1}^{\infty} \frac{1}{x^{4}}dx \lesssim1. \square
 \end{split} \]
\end{remark}
For $0<\lambda< 1$, let $p=\frac{2}{1-\lambda}$ in Theorem 1.2, then we may obtain the boundedness of    $\Cu$ acting from  $L^{2,\lambda}$ to the mean Lipschitz space $\Lambda^{q}_{\frac{1}{q}}$.
\begin{corollary}
Suppose $0<\lambda<1$, $1<p<\infty$ and  $\mu$ is a  finite positive Borel measure on the interval $[0,1)$. Let $X$ be subspace of $\hd$ such that $\Lambda^{p}_{\frac{1}{p}}\subseteq X \subseteq \mathcal {B}$. Then $\Cu: L^{2,\lambda}\rightarrow X$ is bounded if and only if $\mu$ is a $\frac{3-\lambda}{2}$-Carleson measure.
\end{corollary}

\begin{proof of Theorem 1.3}
$(1)\Rightarrow (2)$. Let  $\Lambda_{1}^{1}\subseteq X\subseteq\mathcal {B}$ and $L^{2,\lambda}\subseteq Y\subseteq \mathcal {B}^{\frac{3-\lambda}{2}}$. It is easy to  check that  $g(z)=\log\frac{1}{1-z}\in X$ and
$$\Cu(g)(z)=\sum_{k=0}^{\infty}\mu_{k}\left(\sum_{n=1}^{k}\frac{1}{n}\right)z^{n}.$$
If $\Cu(X)\subseteq Y$, then $\Cu(g)\in Y\subseteq \mathcal {B}^{\frac{3-\lambda}{2}}$. It follows from Lemma \ref{lm2.1} that
$$\sum_{k=1}^{\infty}k\mu_{k}\left(\sum_{n=1}^{k}\frac{1}{n}\right)r^{n}\lesssim \frac{1}{(1-r)^{\frac{3-\lambda}{2}}}, \ \ r\in (0,1).$$
For $K\geq 2$ take $r_{k}=1-\frac{1}{K}$. Since the sequence $\{\mu_{k}\}$ is decreasing,
simple estimations lead us to the following
\[ \begin{split}
K^{\frac{3-\lambda}{2}} &\gtrsim \sum_{k=1}^{\infty}k\mu_{k}\left(\sum_{n=1}^{k}\frac{1}{n}\right)r^{n}_{K}\\
 &\gtrsim \sum_{k=1}^{K}k\mu_{k}\left(\sum_{n=1}^{k}\frac{1}{n}\right)r^{n}_{K}\\
& \gtrsim   \sum_{k=1}^{K}k\mu_{k}\log kr^{n}_{K}\\
& \gtrsim \mu_{K} \sum_{k=1}^{K}k\log k \\
& \asymp  \mu_{K}K^{2}\log K.
  \end{split} \]
Hence $\mu_{K}\lesssim \frac{1}{K^{\frac{1+\lambda}{2}}\log K}$ which implies that $\mu$ is a $1$-logarithmic $\frac{1+\lambda}{2}$-Carleson measure.

$(2)\Rightarrow (1)$. Assume that $\mu$ is a $1$-logarithmic $\frac{1+\lambda}{2}$-Carleson measure. It suffices to show that $\Cu: \mathcal {B}\rightarrow L^{2,\lambda}$ is bounded.
Let $f\in \mathcal {B}$, it is clear that
\[ \begin{split}
|\Cu (f)'(z)|&\leq \int_{0}^{1}\frac{|tf'(tz)|}{|1-tz|}d\mu(t)+ \int_{0}^{1}\frac{|tf(tz)|}{|1-tz|^{2}}d\mu(t)\\
&\leq ||f||_{\mathcal {B}}\int_{0}^{1}\frac{d\mu(t)}{|1-tz|(1-t|z|)}+ ||f||_{\mathcal {B}} \int_{0}^{1}\frac{\log\frac{e}{1-t|z|}}{|1-tz|^{2}}d\mu(t)\\
  \end{split} \]
  This gives
  \[ \begin{split}
 & ||\Cu(f)||_{L^{2.\lambda}}\asymp \sup_{w\in \dd}\left((1-|w|^{2})^{1-\lambda}\int_{\dd}|\Cu(f)'(z)|^{2}(1-|\sigma_{w}(z)|^{2})dA(z)\right)^{\frac{1}{2}}\\
&  \lesssim  ||f||_{\mathcal {B}}\sup_{w\in \dd}(1-|w|^{2})^{\frac{2-\lambda}{2}}\left(\int_{\dd}
  \left(\int_{0}^{1}\frac{d\mu(t)}{|1-tz|(1-t|z|)}\right)^{2}\frac{1-|z|^{2}}{|1-z\overline{w}|^{2}}dA(z)\right)^{\frac{1}{2}}\\
 & +  ||f||_{\mathcal {B}}\sup_{w\in \dd}(1-|w|^{2})^{\frac{2-\lambda}{2}}\left(\int_{\dd}
  \left(\int_{0}^{1}\frac{\log\frac{e}{1-t|z|}d\mu(t)}{|1-tz|^{2}}\right)^{2}\frac{1-|z|^{2}}{|1-z\overline{w}|^{2}}dA(z)\right)^{\frac{1}{2}}\\
 & := E_{1} +E_{2}.
    \end{split} \]
Since $\mu$ is a  $1$-logarithmic $\frac{1+\lambda}{2}$-Carleson measure, by the Minkowski inequality,  Lemma \ref{lm2.4}  and Proposition \ref{pro3.1}, we have
  \[ \begin{split}
 E_{2}&:=||f||_{\mathcal {B}}\sup_{w\in \dd}(1-|w|^{2})^{\frac{2-\lambda}{2}}\left(\int_{\dd}
  \left(\int_{0}^{1}\frac{\log\frac{e}{1-t|z|}d\mu(t)}{|1-tz|^{2}}\right)^{2}\frac{1-|z|^{2}}{|1-z\overline{w}|^{2}}dA(z)\right)^{\frac{1}{2}}\\
  & \leq ||f||_{\mathcal {B}}\sup_{w\in \dd}(1-|w|^{2})^{\frac{2-\lambda}{2}}\int_{0}^{1}\left(\int_{\dd}\frac{\log^{2}\frac{e}{1-|z|}(1-|z|^{2})dA(z)}{|1-tz|^{4}|1-z\overline{w}|^{2}}\right)^{\frac{1}{2}}
d\mu(t)\\
&\asymp ||f||_{\mathcal {B}}\sup_{w\in \dd}\int_{0}^{1}\frac{(1-|w|^{2})^{\frac{2-\lambda}{2}}\log\frac{e}{1-t}}{(1-t)^{\frac{1}{2}}|1-t\overline{w}|}d\mu(t)\\
&\leq  ||f||_{\mathcal {B}}\sup_{w\in \dd}\int_{0}^{1}\frac{(1-|w|^{2})^{\frac{2-\lambda}{2}}\log\frac{e}{1-t}}{(1-t)^{\frac{1}{2}}(1-t|w|)^{\frac{1+\lambda}{2}+\frac{2-\lambda}{2}-\frac{1}{2}}}d\mu(t)\\
&\lesssim ||f||_{\mathcal {B}}.
   \end{split} \]
 Note that   $\mu$ is also a $\frac{1+\lambda}{2}$-Carleson measure. Arguing as the proof of  Theorem 1.1 (the case of $\delta_{p}=0$) we may obtain that
$$E_{1}:=||f||_{\mathcal {B}}\sup_{w\in \dd}(1-|w|^{2})^{\frac{2-\lambda}{2}}\left(\int_{\dd}
  \left(\int_{0}^{1}\frac{d\mu(t)}{|1-tz|(1-t|z|)}\right)^{2}\frac{1-|z|^{2}}{|1-z\overline{w}|^{2}}dA(z)\right)^{\frac{1}{2}}\lesssim ||f||_{\mathcal {B}}.$$
Therefore, we deduce that
    $$||\Cu(f)||_{L^{2,\lambda}} \lesssim   ||f||_{\mathcal {B}} .$$

    The proof is complete.
\end{proof of Theorem 1.3}

\begin{corollary}
Suppose $0<\lambda\leq1$ and  $\mu$ is a  finite positive Borel measure on the interval $[0,1)$. Let $ X$ be subspace
of  $H(\dd)$ such that $\Lambda^{1}_{1}\subseteq X\subseteq \mathcal {B}$. Then $\Cu: X\rightarrow  L^{2,\lambda}$ is bounded if and only if $\mu$ is a $1$-logarithmic$\frac{1+\lambda}{2}$-Carleson measure.
\end{corollary}

\begin{proof of Theorem 1.4}
$(1)\Rightarrow (2)$. Arguing as the   proof of Theorem 1.3 we may obtain that $\mu$ is  a $1$-logarithmic $1$-Carleson measure.

$(2)\Rightarrow (1)$. Suppose  $\mu$ is  a $1$-logarithmic $1$-Carleson measure and  $\Lambda^{1}_{1}\subseteq X,Y\subseteq \mathcal {B}$.
Note that $f\in X\subseteq \mathcal {B}$, by (3.5) we have
\[ \begin{split}
|\Cu(f)''(z)|&\lesssim ||f||_{\mathcal {B}}\int_{0}^{1}\left(\frac{(1-t|z|)^{-2}}{|1-tz|}+\frac{(1-t|z|)^{-1}}{|1-tz|^{2}}+\frac{\log\frac{e}{1-t|z|}}{|1-tz|^{3}}\right)d\mu(t)\\
&\lesssim ||f||_{\mathcal {B}}\left(\int_{0}^{1}\frac{d\mu(t)}{(1-t|z|)^{2}|1-tz|}+\int_{0}^{1}\frac{\log\frac{e}{1-t|z|}d\mu(t)}{|1-tz|^{3}}\right).
  \end{split} \]
 Using (1.2) and  Fubini's theorem, Lemma \ref{lm2.5} and Proposition \ref{pro3.1}, we have
  \[ \begin{split}
 & \ \ \ \ \sup_{0\leq r<1}(1-r)M_{1}(\Cu(f)'',r)\\
  & \lesssim  ||f||_{\mathcal {B}}\sup_{0\leq r<1}(1-r)\int_{0}^{2\pi}
  \int_{0}^{1}\frac{d\mu(t)}{(1-tr)^{2}|1-tre^{i\theta}|}d\theta\\
  & + ||f||_{\mathcal {B}}\sup_{0\leq r<1}(1-r)\int_{0}^{2\pi}\int_{0}^{1}\frac{\log\frac{e}{1-tr}d\mu(t)}{|1-tre^{i\theta}|^{3}}d\theta\\
  &\lesssim ||f||_{\mathcal {B}} \sup_{0\leq r<1} \int_{0}^{1}\frac{1-r}{(1-tr)^{2}}\int_{0}^{2\pi}\frac{d\theta}{|1-tre^{i\theta}|}d\mu(t)\\
& \ \ \ \ +  ||f||_{\mathcal {B}}\sup_{0\leq r<1}\int_{0}^{1}(1-r)\log\frac{e}{1-t}\int_{0}^{2\pi}\frac{d\theta}{|1-tre^{i\theta}|^{3}}d\mu(t)\\
&\lesssim ||f||_{\mathcal {B}} \sup_{0\leq r<1}\int_{0}^{1}\frac{(1-r)\log\frac{e}{1-t}}{(1-tr)^{2}}d\mu(t)\\
&\lesssim ||f||_{\mathcal {B}}.
    \end{split} \]
    This yields  that $\Cu(f)\in\Lambda^{1}_{1}\subseteq Y$.
\end{proof of Theorem 1.4}
Note that the spaces $\Lambda^{p}_{\frac{1}{p}}$ for $1\leq p<\infty$, the spaces $Q_{p}$ for all  $0< p<\infty$  are satisfied the condition in Theorem 1.4.
Therefore, we may obtain  a number of results.

\begin{proof of Theorem 1.5}
$(1)\Rightarrow(2)$. Let $f(z)=(1-z)^{-\frac{1-\lambda_{1}}{2}}$, then Lemma \ref{lm2.6} shows that $f\in L^{2,\lambda_{1}}\subseteq X$. Note that  $\Cu(f)\in Y\subseteq \mathcal {B}^{\frac{3-\lambda_{2}}{2}}$ and
$$\Cu(f)(z)=\sum_{k=0}^{\infty}\mu_{k}\left(\sum_{j=0}^{k}\frac{\Gamma(\frac{1-\lambda_{1}}{2}+j)}{\Gamma(\frac{1-\lambda_{1}}{2})\Gamma(j+1)}\right)z^{k},$$
By the Stirling formula,
$$
\frac{\Gamma(j+\frac{1-\lambda_{1}}{2})}{\Gamma(\frac{1-\lambda_{1}}{2})\Gamma(j+1)}\asymp (j+1)^{-\frac{1+\lambda_{1}}{2}}
$$
for all nonnegative integers $j$. This together with $\{\mu_{k}\}$ is decreasing with $k$ and Lemma \ref{lm2.2}  we deduce that
\[ \begin{split}
1 &\gtrsim n^{-\frac{3-\lambda_{2}}{2}}\sum_{k=1}^{n}k\mu_{k}\left(\sum_{j=0}^{k}\frac{\Gamma(\frac{1-\lambda_{1}}{2}+j)}{\Gamma(\frac{1-\lambda_{2}}{2})\Gamma(j+1)}\right)\\
&\gtrsim n^{-\frac{3-\lambda_{2}}{2}} \sum_{k=1}^{n}k\mu_{k}\left(\sum_{j=0}^{k}(j+1)^{-\frac{1+\lambda_{1}}{2}}\right)\\
&\gtrsim  n^{-\frac{3-\lambda_{2}}{2}} \mu_{n}\sum_{k=1}^{n}k^{\frac{3-\lambda_{1}}{2}}\\
&\gtrsim  \mu_{n}n^{1+\frac{\lambda_{2}-\lambda_{1}}{2}}.
  \end{split} \]
  Lemma \ref{lm2.3} shows that $\mu$ is a $1+\frac{\lambda_{2}-\lambda_{1}}{2}$-Carleson measure.

$(2)\Rightarrow(1)$. It suffices to prove that $\Cu:\mathcal {B}^{\frac{3-\lambda_{1}}{2}}\rightarrow L^{2,\lambda_{2}}$ is bounded. Let $f\in \mathcal {B}^{\frac{3-\lambda_{2}}{2}}$, then  (1.3) and  Lemma \ref{lm2.1}  imply   that
$$|\Cu(f)'(z)|\lesssim ||f||_{\frac{3-\lambda_{1}}{2}}\int_{0}^{1}\frac{d\mu(t)}{(1-t|z|)^{\frac{3-\lambda_{1}}{2}}|1-tz|}.$$
Then arguing as the proof of Theorem 1.1 we can get the desired result.

The proof is complete.
\end{proof of Theorem 1.5}
\begin{corollary}
Suppose   $0<\lambda< 1$, $\mu$ is a  finite positive Borel measure on the interval $[0,1)$. Then $\Cu$ is a bounded operator on $L^{2,\lambda}$ if and only if  $\mu$ is a $1$-Carleson measure.
\end{corollary}

\section*{Data Availability}
No data were used to support this study.

\section*{Conflicts of Interest}
The authors declare that there is no conflict of interest.

\section*{Funding}
The research is  support  by thee Natural Science Foundation of Hunan Province (No. 2022JJ30369).






\begingroup

 \end{document}